\documentclass{amsart}
\usepackage{graphics,verbatim,color,amsthm}

\usepackage[all]{xy}

\theoremstyle{plain}
\newtheorem{theorem}{Theorem}

\newtheorem{lemma}{Lemma}

\theoremstyle{remark}

\theoremstyle{definition}

\def\metrictwo#1{\langle\!\langle #1 \rangle\!\rangle}

\def\normtwo#1{\| #1 \|}

\def\R{\mathbb{R}}

\def\a{\alpha}

\def\g{\gamma}

\def\r{\rho}

\def\w{\omega}

\def\into{\rightarrow}

\DeclareMathOperator{\diag}{diag}

\def\smallskip{\par\vspace{1mm}}
\def\medskip{\par\vspace{2mm}}
\def\bigskip{\par\vspace{3mm}}

\def\fr#1#2{\frac{#1}{#2}}

\def\m#1{\begin{bmatrix}#1\end{bmatrix}}

\newcount\equationcount
\def\thenumber{0}
\def\eq#1{\global\advance\equationcount by 1
   \def\thenumber{\number\equationcount}
                        {$$#1\eqno(\thenumber)$$}}

\begin{document}

\author[R. Moeckel]{Richard Moeckel}
\address{School of Mathematics\\ University of Minnesota\\ Minneapolis MN 55455}
\email{rick@math.umn.edu}

\title[Large potential orbits]{Orbits of the three-body problem\\ with large potential}
\date{March  9, 2026}
\maketitle
Consider the planar three-body problem with masses positive $m_1,m_2,m_3$ position vector  $q(t) = (q_1(t),q_2(t),q_3(t))\in\R^6$.  Let 
$$U(q) = \frac{m_1m_2}{r_{12}}+\frac{m_1m_3}{r_{13}}+\frac{m_2m_3}{r_{23}}$$
where $r_{ij}=|q_i-q_j|$.  Assume that the angular momentum $\w$ is nonzero so that triple collision is impossible.

The goal of this note is to prove the following theorem.
\begin{theorem}\label{th_UK}
Given any constant $K>0$ there are solutions of the negative-energy, planar three-body problem with $U(q(t))\ge K$ for all $t\in\R$.  
\end{theorem}
These solutions will have a single close approach to triple collision.  The configuration will always be a tight binary with $m_1, m_2$ close and the distance from the binary to $m_3$ diverging as $t\into\pm\infty$.  

Motivation for this result came from a question posed in a preprint of Richard Montgomery entitled ``Halfway between heaven and hell''  \cite{Montgomery}.  In that paper, heaven is the zero velocity surface in configuration space where the kinetic energy is zero and the potential energy $U(q)$  is as small as possible for a given energy.  Hell is the singular set where $U(q)$ is infinite.  In between is the virial surface where the relative equilibrium solutions live.  From that perspective, Theorem~\ref{th_UK} shows that there are solutions which stay for all time on the ``hot'' side of the virial surface.  

For most of the paper, we  suppose that the binary collisions are regularized.  Then all of the regularized solutions will exist for all $t\in\R$.  
The proof will show that the set of initial  conditions of the regularized problem leading to solutions of this type has nonempty interior.   Since  the  set  of initial conditions leading to solutions having binary collisions has measure zero and is of Baire first category, it follows that there will still be a nonempty set of such solutions  with no  binary collisions.  So the theorem also holds for nonregularized solutions.

The construction is a refinement of  Birkhoff's arguments in his Dynamical Systems book, \cite{Birkhoff}, Chapter IX.8.  Birkhoff's goal was to prove a result of Sundman which states that for every solution of the  three-body problem with nonzero angular momentum, the size of the triangle formed by the bodies has a positive lower bound.  Triple collision itself is not possible and  Birkhoff's  proof shows that any solution which approaches sufficiently close to triple collision is of the type just described.  In particular, it can have only one very close approach to collision and then is scattered to infinity.  Intuitively, such a tight binary configuration should have a large potential but Birkhoff does not try to get an estimate.  Here, in addition to estimating the potential, we also take the opportunity to study the near approach to triple collision using  McGehee coordinates.

We begin by  eliminating the center of mass using Jacobi coordinates $x=(x_1,x_2)\in\R^4$
$$x_1=q_2-q_1\qquad  x_2=q_3-\nu_1 q_1-\nu_2q_2.$$
Then the Lagrangian is $L=T(\dot x)+U(x)$ where
$$\begin{aligned}
T &= \fr12\mu_1|\dot x_1|^2+\fr12\mu_2|\dot x_2|^2\\
U(x) &= \fr{m_1m_2}{r_{12}}+\fr{m_1m_3}{r_{13}} + \fr{m_2m_3}{r_{23}} \\
r_{12}&=|x_1|\quad r_{13}= |x_2+\nu_2 x_1|\quad r_{23} =|x_2-\nu_1x_1|
\end{aligned}
$$
where the mass constants are
$$\mu_1= \fr{m_1m_2}{m_1+m_2} \quad \mu_2= \fr{(m_1+m_2)m_3}{m_1+m_2+m_3}\quad \nu_i=\fr{m_i}{m_1+m_2}.$$
The Euler-Lagrange differential equations are
\begin{equation}
\begin{aligned}\label{eq_odeJacobi}
\mu_1 \ddot x_1&= U_{x_1}\\
\mu_2 \ddot x_2&= U_{x_2}\\
\end{aligned}
\end{equation}
Solutions of the Euler-Lagrange equations preserve the energy $T-U = h<0$.
The angular momentum $\w=\mu_1 x_1\wedge \dot x_1 +\mu_2 x_2\wedge\dot x_2$ is also a constant of motion.

Introduce a {\em mass norm} and {\em mass metric} in $\R^4$ by
$$\normtwo{v}= \mu_1|v_1|^2+\nu_2|v_2|^2\qquad\metrictwo{v,w}= \mu_1 v_1\cdot w_1+\mu_2 v_2\cdot w_2$$
where $v_i, w_i\in\R^2$.  If $M$ is the $4\times 4$ mass matrix $M=\diag(\mu_1,\mu_1,\mu_2,\mu_2)$ then $\metrictwo{v,w}=v^TMw$.
Using these definitions, the kinetic energy is $T=\fr12\normtwo{\dot x}^2$ and
$$I = r^2=\normtwo{x}^2  =\mu_1|x|^2+\mu_2|y|^2$$
 is the moment of inertia.   Finally, if $J$ is the $4\times 4$ block-diagonal matrix with blocks $\m{0&-1\\1&0}$, the angular momentum is
 $$\w =\metrictwo{x,J\dot x}.$$

\section{Near Triple Collision}
Next introduce McGehee coordinates to blow up the triple collision.
With $r=\sqrt{I}$ as above, define a normalized configuration variable $s= x/r$ and a rescaled velocity variable 
$z=\sqrt{r}\, \dot x$.  In addition, introduce a new time variable such that $dt=r^\fr32 \,d\tau$.  Then the blown-up differential equations are
\begin{equation}\label{eq_odeblowup}
\begin{aligned}
r' &= v\,r\\
s'&= z-vs \\
z'&=M^{-1}\nabla V(s)+\fr12 v z.
\end{aligned}
\end{equation}
where $v=\metrictwo{s,z}$.  
The energy equation becomes
\begin{equation}\label{eq_energyblowup}
\fr12\normtwo{z}^2-V(s) = rh.
\end{equation}
The shape potential $V(s)$ is the same as $U$ with $x_i$ replaces by $s_i=x_i/r$.  By homogeneity we have
$$U(x)=\fr1r V(s).$$

The variable $v = \metrictwo{s,z}$ will play an important role.  Since $r'=vr$ we have the size variable $r(\tau)$ increasing when $v>0$ and decreasing when $v<0$.
From (\ref{eq_odeblowup}) and (\ref{eq_energyblowup}) we get
\begin{equation}\label{eq_vprime}
v' =\fr12\normtwo{z^2} -\fr12v^2+rh.
\end{equation}
Next we develop some inequalities involving the angular momentum.   From the blow-up of coordinates we have
$$r^{-\fr12} \w = \mu_1 s_1\wedge z_1 +\mu_2 s_2\wedge z_2 = \metrictwo{s,Jz}.$$
Note that $s, Js$ are orthogonal unit vectors with respect to the mass metric.  The orthogonal projection of $z$ onto the plane they span is
$\metrictwo{z,s}s +\metrictwo{z,Js}Js$.  It follows that
\begin{equation}\label{eq_Testimate}
\normtwo{z}^2\ge \metrictwo{z,s}^2 +\metrictwo{z,Js}^2 = v^2 +\fr{\w^2}{r}.
\end{equation}
Then (\ref{eq_vprime}) gives
\begin{equation}\label{eq_vprimeestimate}
v' \ge\fr{\w^2}{2r}+rh.
\end{equation}
It follows from this that $v(\tau)$ is monotonically increasing near triple collision, more precisely, whenever $r^2<\fr{\w^2}{2|h|}$.

Consider a solution $\g(\tau)$ with initial condition $r(0)=r_0$ with $r_0^2<\fr{\w^2}{2|h|}$ and $v(0)=0$.  Then since $v'(0)>0$,  $v(\tau)$ changes sign from negative to positive and $r(\tau)$ has a nondegenerate local minimum at $\tau=0$.  The goal will be to show that if $r_0>0$ sufficiently small, then along $\g(\tau)$ we have  $r(\tau)\into\infty$ as $\tau\into\pm\infty$ and that
$U(\tau)= \fr{1}{r(\tau)}V(s(\tau))\ge K$ for all time.  It suffices to consider $\tau\into\infty$.

Following Sundman and Birkhoff, we introduce the following useful function (which Birkhoff calls $H$):
$$F = v^2-2rh+\fr{\w^2}r.$$
Then using (\ref{eq_vprimeestimate}) as long as $v\ge 0$ we have
$$F' = 2vv'-2hr'-\fr{\w^2}{r^2}r' \ge 2v(\fr{\w^2}{2r}+rh)-2hrv-\fr{\w^2}{r}v = 0.$$
Thus $F(\tau)$ is non-decreasing along any solution with $v(\tau)\ge 0$.  Since $h<0$, the initial value of $F$ is
$$F_0(r_0) = F(0) = \fr{\w^2}{r_0}+2|h|r_0.$$
Note that by choosing $r_0$ sufficiently small we can make $F_0(r_0)$ arbitrarily large.

The size $r(\tau)$, $\tau>0$, will continue to increase as long as $v(\tau)>0$.  The inequality $F(\tau)\ge F_0$ shows that $v(\tau)$ cannot reach zero as long as
$$\fr{\w^2}{r}+2|h|r \ge \fr{\w^2}{r_0}+2|h|r_0.$$
This inequality simplifies nicely to
\begin{equation}\label{eq_rbound1}
r\ge \fr{\w^2}{2|h|r_0}.
\end{equation}
Thus by choosing $r_0$ sufficiently small we can guarantee that $r(\tau)$ will continue to monotonically increase to arbitrarily large values.

In addition to getting a large $r$ we also want large values for $v$ and $U$.  Using the energy equation and (\ref{eq_Testimate}) we get
$$\normtwo{z}^2=2V(s)+2rh \ge v^2 +\fr{\w^2}{r}$$
which can be rearranged to show
$$2V(s(\tau)) \ge F(\tau).$$
Since $F(\tau)\ge F_0$ we can arrange that $V(s(\tau))$ be arbitrarily large along our orbit.  The function $F$ also allows us to bound $v(\tau)$ using
$$v^2 \ge F_0 - \fr{\w^2}{r}-2|h|r.$$
\begin{lemma}\label{lemma_r1}
Let $\g(\tau)$ be a solution with  $r(0)=r_0$ with $r_0^2<\fr{\w^2}{2|h|}$ and $v_0=0$ and choose any $r_1<\fr{\w^2}{2|h|r_0}$.  Then there is a time interval $[0,\tau_1]$ such that $r(\tau)$ is strictly increasing with $r(\tau_1)=r_1$ and for all $\tau\in [0,\tau_1]$ we have
$$V(\tau)\ge  \fr{\w^2}{2r_0}+|h|r_0\qquad U(\tau)=\fr{1}{r(\tau)}V(\tau) \ge  \fr{\w^2}{2r_0r_1}+|h|\fr{r_0}{r_1}.$$
Moreover
\begin{equation}\label{eq_v1}
v_1^2=v(\tau_1)^2\ge(r_1-r_0)\left[\fr{\w^2}{2r_0r_1}-2|h|\right].
\end{equation}
\end{lemma}
\begin{proof}
Since $r(0)=r_0$ with $r_0^2<\fr{\w^2}{2|h|}$ and $v_0=0$, $r(0)$ is a local minimum and $r(\tau)$ begins to increase with $v(\tau)>0$.  Since  $r_1<\fr{\w^2}{2|h|r_0}$, $v(\tau)>0$ and $r(\tau)$ increases at least until it reaches $r_1$ at some time $\tau_1$.  Meanwhile $V(\tau)\ge \fr12 F_0$ which is the estimate in the lemma.  Since $r(\tau)$ was increasing, the extra factor $1/r$  in $U(\tau)$ is minimized at $ r_1$.  Finally $F(\tau_1)\ge F_0$ implies
$$v^2(\tau_1) \ge \fr{\w^2}{r_0}+2|h|r_0- \fr{\w^2}{r_1}-2|h|r_1=(r_1-r_0)\left[\fr{\w^2}{2r_0r_1}-2|h|\right]$$
as claimed.
\end{proof}

In order to get all of the required quantities to be large at $\tau_1$ we will choose a value for $r_1$ which is large but not too large.  For example suppose we take
$r_1 = \fr{\w^2}{2\sqrt{r_0}}$.  For $r_0$ sufficiently small, this satisfies the hypotheses of the lemma.  Then we have the estimates
\begin{equation}\label{eq_UVestimate}
V(\tau)\ge \fr{\w^2}{2r_0}\qquad U(\tau)\ge \fr{1}{\sqrt{r_0}}+\fr{2|h|}{\w^2}r_0^\fr32 \ge \fr{1}{\sqrt{r_0}}\qquad \tau\in [0,\tau_1]
\end{equation}
and 
$$v_1^2\ge \fr{1}{r_0}\left( \fr{\w^2}{2}-\w^2|h|\sqrt{r_0} -r_0^\fr32 +2|h|r_0^2  \right).$$
So for $r_0$ sufficiently small we can arrange for $r_1, v_1,V(\tau), U(\tau)$, $\tau\in [0,\tau_1]$ to be as large as we please.

\section{Off to infinity}
In this section we will show that we can continue the near-collision orbits from the last section in such a way that $r(\tau)$ increases monotonically to infinity while $U(x(\tau))$ remains large.  Once again we are guided by Birkhoff's argument.  Let $x=(x_1,x_2)$ be the Jacobi variables.  Then $|x_1|=r_{12}$ and, as in Birkhoff, let $\r=|x_2|$.  We have
$$r^2=\mu_1 r_{12}^2 +\mu_2\r^2.$$
 
In this section we will use the ordinary time variable $t$.   Let $t_1$ be the time corresponding to the final time $\tau_1$ in the last section. Suppose that at time $t_1$, $r_{12}$ is the smallest of the  mutual distances.  Then 
 $$U(t_1)\ge \fr{\a}{r_{12}(t_1)}\qquad r_{12}(t_1)\le \fr{\a}{U(t_1)}$$
 where 
 $$\a=m_1m_2+m_1m_3+m_2m_3.$$
Recall that by choosing $r_0$ sufficiently small we can make $U(\tau_1)$ and $r(t_1)$ arbitrarily large.   It follows that we can make $r_{12}(t_1)$ arbitrarily  small  and therefore
$\r(t_1)$ arbitrarily large.   We will  show below that we can also arrange that $\dot \r(t_1)$ be large.  Then we can use the following lemma  to  complete  the  proof of Theorem~\ref{th_UK}.

\begin{lemma}\label{lemma_Birkhoff}
If  the initial values of $\r(t_1)$ and $\dot \r(t_1)$ are sufficiently large, then $\r(t)$ increases monotonically to infinity  for $t\in[t_1,\infty)$.  Moreover, we  can arrange for the value of  
$\dot \r(t)$ to remain above any given positive constant.
\end{lemma}
\begin{proof}
This proof is essentially that in Birkhoff's book.  Since $U(\tau)\ge |h|$, the minimal distance $r_{12}$ satisfies $r_{12}(t_1)\le \a/|h|$.  The other distances $r_{13},r_{23}$ satisfy
$\r-r_{12} \le r_{i3} \le \r+r_{12}$ by the triangle inequality.  As long as $\r(t)>2\a/|h|\ge 2r_{12}$, $r_{12}(t)$ will remain the smallest distance.  Since $\r=|x_2|$ we have
$\r\ddot\r+ \dot\r^2 = x_2\cdot \ddot x_2 + |\dot x_2|^2$.  Since $\dot\r^2\le |\dot x_2|^2$ we have
$$
\begin{aligned}
\r\ddot \r &\ge x_2\cdot \ddot x_2 = -\fr{m_1m_3 x_2\cdot (x_2+\nu_2x_1)}{r_{13}^3}  -\fr{m_2 m_3 x_2\cdot (x_2-\nu_1 x_1)}{r_{23}^3}\\
&\ge -\fr{m_1m_3 \r}{r_{13}^2} -\fr{m_2 m_3 \r}{r_{23}^2}.
\end{aligned}$$
Since $r_{i3}\ge \r-r_{12}$ and $r_{12}\le \fr12 \r$ and
\begin{equation}\label{eq_rhoddot}
\ddot \r \ge -\fr{4(m_1+m_2)m_3}{\r^2}.
\end{equation}
Now we can compare the behavior of $\r(t)$ with that of a Kepler problem on a line.  Let
$$G = \fr12 \dot\r^2 - \fr{4(m_1+m_2)m_3}{\r}$$
be the Keplerian energy.
From (\ref{eq_rhoddot}) we find that $\dot G(t)\ge 0$ so $G(t)\ge G(t_1)$ for $t\in [t_1,\infty)$.  Once we have shown that we can make $\r(t_1), \dot \r(t_1)$ arbitrarily large, it follows that we can make $G(t_1)$ arbitrarily large.  Therefore $\dot\r(t) \ge \sqrt{2G(t_1)}$ remains arbitrarily large for $t\in [t_1,\infty)$.
\end{proof}

We already know that by choosing $r_0$ sufficiently small, we can make $\r(t_1)$ arbitrarily large.  We will now show that the same is true for $\dot\r(t_1)$.  First note that
\begin{equation}\label{eq_rrdot}
r\dot r = \mu_1 r_{12}\dot r_{12}+\mu_2\r \dot\r.
\end{equation}
The first term here is can be made arbitrarily small by choice of $r_0$.  To see this, first note that  
$$\fr12 \mu_1\dot r_{12}^2 \le \fr12 \mu_1|\dot x_1|^2 \le \fr12\normtwo{\dot x}^2 \le U(x).$$
So
$$\fr12 \mu_1r_{12}\dot r_{12}^2 \le  r_{12}U(x) =  m_1m_2+\fr{m_1m_3r_{12}}{r_{13}}+\fr{m_2m_3r_{12}}{r_{23}}\le  \a$$
where the last  inequality  holds because $r_{12}$ is the smallest distance.  This shows $r_{12}\dot r_{12}^2$ is bounded and since $r_{12}$ can be made arbitrariiy small, 
so can $(r_{12}\dot r_{12})^2$.

Dividing (\ref{eq_rrdot}) by $r$ shows that  
$$\dot r  = \fr{\mu_1 r_{12}\dot r_{12}}{r}+\fr{\mu_2\r}{r} \dot\r.$$  
Since $r'=vr$ we have $\dot r ^2 = \fr{v^2}{r}$.  At $t=t_1$ we have $r_1 = \fr{\w^2}{2\sqrt{r_0}}$ while from (\ref{eq_v1}) we find that $v_1^2$ is of order $1/r_0$.  So the ratio defining $\dot r(t_1)$ can still be made arbitrarily large.    
Since  the first term is small and the ratio $\fr{\mu_2\r}{r}$ is bounded  below, we can make  $\dot \r$ arbitrarily  large, as  claimed.

\begin{proof}[Proof of Theorem~\ref{th_UK}]
Let $K>0$ be given.  Choose any $r_0$ with $r_0^2<\fr{\w^2}{2|h|}$ and  $r_0\le \fr{1}{K^2}$ and let $r_1 = \fr{\w^2}{2\sqrt{r_0}}$.  Applying Lemma~\ref{lemma_r1} gives a solution with $r(0)=r_0$, $v(0)=0$ and a time interval $[0,\tau_1]$ over which $r(\tau)$ increases monotonically to $r_1$.  By (\ref{eq_UVestimate}), the potential satisfies 
$$U(\tau)\ge \fr{1}{\sqrt{r_0}}\ge K\qquad \tau\in[0,\tau_1].$$

Switching to the time variable $t$ we use the results of this section to continue the solution over $[t_1,\infty)$.  Taking $r_0$ smaller if necessary, we can arrange that $\r(t_1)$ and $\dot\r(t_1)$ are arbitrarily large.   Lemma~\ref{lemma_Birkhoff} shows that we can continue our solution in such a way that $\r(t)$ monotonically increases to infinity and such that
$\dot \r(t)\ge \sqrt{2K/\mu_2}$ for all  $t\in [t_1,\infty)$.  It follows that 
$$U(t) = T + |h| \ge T\ge \fr12 \mu_2|\dot x_2(t)|^2 \ge \fr 12 \mu_2\dot\r(t)^2\ge K.$$

Thus every orbit with $r(0)=r_0$, $v(0)=0$ and $r_0$ sufficiently small satisfies $U(t)\ge K$ for $t\ge 0$.  Reversing time shows that the same thing is true for $t\le 0$.  Therefore if $r_0$ is sufficiently small, we actually have $U(t)\ge K$ for all $t\in\R$ as required.  Moreover, these  orbits represent  an open set  of initial conditions for the regularized problem.

Now the set of initial conditions leading to binary collision is of measure zero and Baire first category.   This is well-known and can be seen as follows.  Because of the regularization, there is a neighborhood of the binary collision set such that the set of initial conditions leading to binary collision without leaving the neighborhood forms a submanifold of codimension one, which is clearly nowhere dense and of measure zero.  Using the regularized flow to extend this submanifold forward and backward in time, we can express the set of all initial conditions leading to binary collision as a countable union of such sets.

Thus within the open set of regularized initial conditions whose orbits satisfy Theorem~\ref{th_UK}, there will be many whose orbits are collision-free. 
\end{proof}

\bibliographystyle{amsplain}

\begin{thebibliography}{20}

\bibitem{Birkhoff} G.D. Birkhoff,  {\sl Dynamical Systems}, American Mathematical Society Colloquium Publications, vol. 
\textbf{9}, 1927.  

\bibitem{Montgomery} R. Montgomery,  {\sl Halfway between heaven and hell}, preprint  2026.

\bibitem{Sundman} K.F. Sundman,  {\sl Memoire sur le probl\`eme des trois corps}, Acta math. 36, 105--179, 1913.


\end{thebibliography}

\end{document}